\documentclass[a4paper, notitlepage]{article}

\usepackage{etex}
\usepackage[latin1]{inputenc}
\usepackage[english, british]{babel}
\usepackage{calc,amssymb,mathtools}
\mathtoolsset{mathic}
\usepackage{amsthm,stmaryrd,units,enumitem,tabularx,booktabs,aliascnt,hyperref}
\hypersetup{breaklinks=true,colorlinks=true,linkcolor=black,anchorcolor=black,citecolor=black,filecolor=black,menucolor=black,pagecolor=black,urlcolor=black}
\usepackage[alphabetic,non-sorted-cites]{amsrefs}
\usepackage[all]{xy}
\entrymodifiers={+!!<0pt,\fontdimen22\textfont2>}
\SelectTips{cm}{10}

\chardef\bslash=`\\

\theoremstyle{plain}
   \newaliascnt{secali}{section}\newtheorem{thm}{Theorem}[secali]    \aliascntresetthe{secali}  
   \newtheorem*{thm*}{Theorem}
   \newaliascnt{cor}{thm}     \newtheorem{cor}[cor]{Corollary}       \aliascntresetthe{cor}     
   \newtheorem*{cor*}{Corollary}
   \newaliascnt{lem}{thm}     \newtheorem{lem}[lem]{Lemma}           \aliascntresetthe{lem}     
   \newcommand*{\theoremprime}{Theorem$^{\text{a}}$}
   \newtheorem*{thmprime}{\theoremprime}
      \newtheorem*{Bousfield-lem}{Bousfield's Lemma}
   \newtheorem*{Hodgkin-thm}{Hodgkin's Theorem}
   \newtheorem*{Steinberg-thm}{Steinberg's Theorem}
   \newtheorem*{obs*}{Observation}
   \newtheorem*{lem*}{Lemma}
   \newaliascnt{prop}{thm}    \newtheorem{prop}[prop]{Proposition}   \aliascntresetthe{prop}    
\theoremstyle{definition}
   \newaliascnt{defn}{thm}        \aliascntresetthe{defn}    
   \newtheorem*{defn*}{Definition}
\theoremstyle{remark}
   \newaliascnt{example}{thm} \newtheorem{example}[example]{Example} \aliascntresetthe{example} 
   \newtheorem*{example*}{Example}
   \newaliascnt{rem}{thm}     \newtheorem{rem}[rem]{Remark}          \aliascntresetthe{rem}     
   \newtheorem*{rem*}{Remark}
\addto\extrasbritish{%
}

\numberwithin{equation}{section}

\newcommand{\Z}{\mathbb{Z}}
\newcommand{\R}{\mathbb{R}}
\newcommand{\C}{\mathbb{C}}

\renewcommand{\H}{\mathbb{H}}

\DeclareMathOperator{\id}{id}
\DeclareMathOperator{\rank}{rank}
\renewcommand{\rank}{\mathrm{rk}}

\renewcommand{\bar}{\overline}

\renewcommand{\tilde}{\widetilde}

\newcommand{\odd}{\text{odd}}

\newcommand{\cat}[1]{\mathcal{#1}}

\DeclareMathOperator{\im}{im}
\newcommand{\under}{\ensuremath{\backslash}}
\newcommand{\leftexp}[2]{{\vphantom{#2}}^{#1}{#2}}

\DeclareMathOperator{\point}{\text{point}}

\newcommand{\factor}[2]{\left.\raisebox{.1em}{\ensuremath{#1}}\middle/\raisebox{-.1em}{\ensuremath{#2}}\right.}
\newcommand{\mm}[1]{\left(\begin{smallmatrix}#1\end{smallmatrix}\right)}

\newenvironment{smalldiagram}{\begin{equation*}\SelectTips{cm}{10}}{\end{equation*}}
\newcolumntype{C}{>{$}c<{$}}
\newcolumntype{L}{>{$}l<{$}}

\DeclareMathOperator{\Kgroup}{K}
        \newcommand{\K}{\ensuremath{\Kgroup}}
\DeclareMathOperator{\KOgroup}{KO}

        \newcommand{\KO}{\ensuremath{\KOgroup}}

\DeclareMathOperator{\Wgroup}{W}
        \newcommand{\W}{\Wgroup}
        
        \DeclareMathOperator{\GWgroup}{GW}
        \newcommand{\GW}{\GWgroup}

\newcommand{\vb}[1]{{\mathcal{#1}}}
\newcommand{\dual}{\vee}

\renewcommand{\star}{\ensuremath{\ast}}

\newcommand{\ideal}[1]{\mathfrak{#1}}
\newcommand{\RepC}{\mathrm{R}}                        
\newcommand{\RepR}{\mathrm{RO}}                       
\newcommand{\RepH}{\mathrm{RSp}}                      

\newcommand{\RepRH}{\mathrm{RO}^{0,4}}                

\newcommand{\aC}{\alpha}                              
\newcommand{\baraC}{\bar\alpha}
\newcommand{\aR}{\alpha_{\mathrm{O}}}                 
\newcommand{\aRH}{\alpha_{\mathrm{O}}^{0,4}}          
\newcommand{\aO}{\alpha_{\mathrm{O}}}
\newcommand{\baraO}{\bar\alpha_{\mathrm O}}
\newcommand{\baraRH}{\bar\alpha_{\mathrm O}^{0,4}}

\newcommand{\keridealC}[1]{\ideal a(#1)}              
\newcommand{\keridealR}[1]{\ideal a_{\mathrm O}^0(#1)}
\newcommand{\keridealH}[1]{\ideal a_{\mathrm O}^4(#1)}
\newcommand{\keridealO}{\ideal a_{\mathrm{O}}}
\newcommand{\keridealRH}[1]{\ideal a_{\mathrm{O}}^{0,4}(#1)}

\newcommand{\reduced}[1]{\widetilde{#1}}

\newcommand{\deffont}[1]{\textbf{#1}}

\newcommand{\GrothendieckWitt}{Grothen\-dieck-Witt }
\newcommand{\ie}{\mbox{i.\thinspace{}e.\ }}

\newcommand{\cf}{\mbox{c.\thinspace{}f.\ }}
\begin{document}
\title{KO-Rings of Full Flag Varieties%
\footnote{First published in Trans.\ Amer.\ Math.\ Soc.\ {\bf 367} (2015), 2997--3016, published by the American Mathematical Society}
}
\author{Marcus Zibrowius%
 \thanks{Bergische Universit{\"a}t Wuppertal,
  Gau{\ss}stra{\ss}e 20,
  42119 Wuppertal,
  Germany}
}
\vspace{-2cm}
\maketitle
\begin{abstract}%
\noindent We present type-independent computations of the \(\KO\)-groups of full flag varieties, \ie of quotient spaces \(G/T\) of compact Lie groups by their maximal tori. Our main tool is the identification of the Witt ring, a quotient of the \(\KO\)-ring, of these varieties with the Tate cohomology of their complex \(\K\)-ring. The computations show that the Witt ring is an exterior algebra whose generators are determined by representations of \(G\).
\end{abstract}
\tableofcontents
\setlength{\parindent}{0pt}
\addtolength{\parskip}{3pt}
\addtolength{\topsep}{3pt}
\thispagestyle{empty}

\section*{Introduction}
The complex \(\K\)-ring of a homogeneous space \(X\) of the form \(G/H\), where \(G\) is a compact Lie group and \(H\) is a subgroup of maximal rank, has a classical description in terms of the representation rings of \(G\) and \(H\). In particular, every stable isomorphism class of complex vector bundles over \(X\) lies in the image of a natural ring homomorphism
\[
   \aC\colon\RepC(H)\rightarrow\K^0(X)
\]
that was already studied by Atiyah and Hirzebruch in their seminal paper \cite{AtiyahHirzebruch:VectorBundles}.

\begin{samepage}
For real vector bundles, the situation is far more subtle. In particular, while it is straight-forward to define an analogous ring homomorphism
\[
   \aR\colon\RepR(H)\rightarrow\KO^0(X),
\]
it is easy to see that it is almost never surjective (\autoref{prop:coker-aO}).
\end{samepage}

On the whole, our current understanding of the \(\KO\)-groups of homogeneous spaces is still rather patchy. The problem of computing these groups is just as old as the corresponding problem for complex \(\K\)-groups, yet no general solution is known to date. More recent progress includes a series of papers of Kishimoto, Kono, Ohsita and Yagita computing the \(\KO\)-groups of all full flag varieties, \ie of the homogeneous spaces of the form \(G/T\), where \(G\) is as above and \(T\) is a maximal torus \citelist{\cite{KKO:Flags}\cite{KO:ExceptionalFlags}\cite{Yagita:W-of-G}}. As in earlier work \citelist{\cite{Fujii:P},\cite{KonoHara:Gr},\cite{KonoHara:HSS}}, the main strategy is to show that the Atiyah-Hirzebruch spectral sequence collapses after the second differential, allowing a computation of the two-torsion of the \(\KO\)-groups via the second Steenrod square.
However, the arguments needed to put this strategy into practice are intricate.
In particular, so far it has been possible to establish that the spectral sequence collapses only on a case-by-case basis,
and the arguments used for the Lie groups of classical types in \cite{KKO:Flags}, for $E_6$, $F_4$ and $G_2$ in \cite{KO:ExceptionalFlags} and for $E_7$ and $E_8$ in \cite{Yagita:W-of-G} each have different flavour. No relationship between vector bundles on \(G/T\) and representations of \(G\) is apparent from these calculations.

In this paper, we present a more conceptual, type-independent computation of the \(\KO\)-groups of all full flag varieties based on the known description of their \(\K\)-rings. Our approach focuses on the ``Witt ring'' of \(X\), for which we give the following ad-hoc definition: Let \(r_i\colon\K^{2i}(X)\rightarrow\KO^{2i}(X)\) denote the realification maps. We define the \deffont{Witt groups} and the (total) \deffont{Witt ring} of \(X\) as the cokernels of these maps and their direct sum, respectively:
\begin{align*}
  \W^i(X) &:= \KO^{2i}(X)/r_i \\
  \W^*(X) &:= \bigoplus_{i\in\Z/4} \KO^{2i}(X)/r_i
\end{align*}
We will seek to justify our terminology at the end of this introduction. First, let us mention the key points that make these quotients interesting:
\begin{itemize}
\item The Witt groups capture the two-torsion of the \(\KO\)-groups of \(X\). In particular, since the free parts of the \(\KO\)-groups may easily be determined by cell-counting, a description of the Witt groups leads to a full additive description of \(\KO^*(X)\) (\autoref{lem:KO-structure}).
\item As a quotient ring of \(\KO^{\text{even}}(X)\), the Witt ring \(\W^*(X)\) provides a first approximation to the ring structure of \(\KO^*(X)\). In fact, it completely describes the products \(\KO^{\text{odd}}(X)\otimes\KO^{\text{odd}}(X)\rightarrow\KO^{\text{even}}(X)\) (\autoref{rem:KO-ring-structure}).
\item The ring \( \W^0(X) \) detects all non-homogeneous vector bundles: elements of \(\KO^0(X)\) not contained in the image of \(\aR\). More precisely, the cokernel of \(\aR\) may be identified with a quotient of \(\W^0(X)\) (\autoref{prop:coker-aO}).
\item The Witt ring is computable: it follows from an observation of Bousfield that \(\W^*(X)\) may be identified with the Tate cohomology of the complex \(\K\)-ring \(\K^0(X)\) (see \autoref{sec:Witt-is-Tate}).
\end{itemize}

Theorem~6.7 of \cite{Yagita:W-of-G}, summarizing the existing calculations for full flag varieties, shows that their Witt rings are exterior algebras on generators of odd degrees; in each case, the explicit number of generators in each degree may be extracted from one of the papers mentioned. The independent computations presented here yield the following concise formulation:
\begin{thm*}
Let \( G \) be a simply-connected compact Lie group.
The Witt ring of \( G/T \) is an exterior algebra on \( b_\H \) generators of degree \( 1 \) and \(\frac{b_\C}{2} + b_\R \) generators of degree \( 3 \), where \(b_\C\), \(b_\R\) and \(b_\H\) denote the numbers of basic representations of \(G\) of complex, real and quaternionic type, respectively.
\end{thm*}
This is \autoref{thm:main} below. Precise definitions of \(b_\C, b_\R\) and \(b_\H\) are given in \autoref{sec:reps}. Their values for simple \(G\), taken from \cite{Davis:Rep}*{Table~3.1}, are displayed in \autoref{table:bbb}.

\begin{table}
\begin{tabular*}{\textwidth}{@{\extracolsep{\fill}}CCLCCCLC}
\toprule
& \text{Type}    & \quad G            & b_\C    & b_\R          & b_\H          &  & \\
\midrule
& A_n            &\mathrm{SU}(n+1)    & n-1     & -             & 1             & \text{if } n\equiv 1 \mod 4 \\
& {}^{(n\geq 1)} &                    & n-1     & 1             & -             & \text{if } n\equiv 3 \mod 4 \\
&                &                    & n       & -             & -             & \text{if } n\text{ is even} \\
\rule{0pt}{1.5em}
& B_n            &\mathrm{Spin}(2n+1) & -       & n-1           & 1             & \text{if } n\equiv 1,2 \mod 4\\
& {}^{(n\geq 2)} &                    & -       & n             & -             & \text{if } n\equiv 0,3 \mod 4\\
\rule{0pt}{1.5em}
& C_n            &\mathrm{Sp}(n)      & -       & (n-1)/2       & (n+1)/2       & \text{if } n\text{ is odd}  \\
& {}^{(n\geq 3)} &                    & -       & n/2           & n/2           & \text{if } n\text{ is even} \\
\rule{0pt}{1.5em}
& D_n            &\mathrm{Spin}(2n)   & -       & n             & -             & \text{if } n\equiv 0 \mod 4\\
& {}^{(n\geq 4)} &\hfill              & -       & n-2           & 2             & \text{if } n\equiv 2 \mod 4\\
&                &                    & 2       & n-2           & -             & \text{if } n\text{ is odd} \\
\rule{0pt}{1.5em}
& E_6            & \quad (E_6)        & 4       & 2             & -             & \\
& E_7            & \quad (E_7)        & -       & 4             & 3             & \\
& E_8            & \quad (E_8)        & -       & 8             & -             & \\
\rule{0pt}{1.5em}
& F_4            & \quad (F_4)        & -       & 4             & -             & \\
\rule{0pt}{1.5em}
& G_2            & \quad  (G_2)       & -       & 2             & -             & \\
\bottomrule
\end{tabular*}
\caption{\protect\centering The numbers of basic representations for the simply-connected simple compact Lie groups \(G\) of different types}\label{table:bbb}
\end{table}

\newcommand{\Walg}{\W}
Our justification for the above terminology is based on the following connection with algebraic geometry.
In \cite{Balmer:TWGI}, Balmer constructs a four-periodic cohomology theory \(\Walg^*\) on algebraic varieties such that, when \(F\) is a field of characteristic not two, \(\Walg^0(\text{Spec}(F))\) agrees with the classical Witt group of \(F\). In \cite{Me:WCCV}, we show that for complex flag varieties Balmer's Witt ring agrees with the ``topological'' Witt ring defined above (see also \autoref{rem:Witt}).
In particular, the theorem immediately translates into a description of Balmer's Witt ring of a full complex flag variety.

In fact, more is true.  We have set up the computations presented here in such a way that they can be done directly in the algebraic setting.   This allows us to generalize the theorem to full flag varieties over any algebraically closed field of characteristic not two:
\begin{thmprime}
Let \(G\) be a simply-connected semi-simple algebraic group over an algebraically closed field of characteristic not two, and let \(B\subset G\) be a Borel subgroup. The graded Witt ring (in the sense of Balmer) \(\W^*(G/B)\) is an exterior algebra on \( b_\H \) generators of degree \( 1 \) and \(\frac{b_\C}{2} + b_\R \) generators of degree \( 3 \), where \(b_\C\), \(b_\R\) and \(b_\H\) are the numbers of non-self-dual, symmetric and anti-symmetric basic representations of \(G\).
\end{thmprime}
In this generality, we believe the result to be new.
Complete algebraic computations of Witt groups of flag varieties are in fact still only known in special cases, for example for projective spaces \cite{Walter:PB}, (split) quadrics \cite{Nenashev:Q} and Grassmannians \cite{BalmerCalmes:Gr}.

That said, our main focus in this article will be on the topological situation.  In particular, some auxiliary results needed for the proof of {\theoremprime} in Section~\ref{sec:W-of-G/T} will simply be quoted from our article \cite{Me:Vanishing},  in which we apply similar ideas to ``twisted'' Witt groups of general flag varieties.

\paragraph{Notation and conventions.} Throughout this paper, a (topological) \deffont{space} will be a space of the homotopy type of a finite CW complex.  We do not deviate from the standard convention that \deffont{simply-connected} spaces are in particular path-connected when referring to Lie groups.  The \(\K\)- and \(\KO\)-groups of a space will be thought of as graded by \(\Z/2\) and \(\Z/8\), respectively. In particular, periodicity isomorphisms are never mentioned explicitly. The notation
\[\KO^{0,4}(X)\]
will indicate the direct sum \(\KO^0(X)\oplus\KO^4(X)\), and similarly for other graded groups and indexes.

\section{Witt rings and Tate cohomology}\label{sec:W-and-Tate}
\subsection{Witt rings capture two-torsion}\label{sec:W-captures-2}
In this section, we summarize a few lemmas concerning the additive structure of the \(\KO\)-groups of flag varieties on which all existing computations rely. To be specific, we use the term \deffont{flag variety} to refer to a homogeneous space of the form
\[
   G/H,
\]
where \( G\) is a semi-simple compact Lie group and \( H\) is a centralizer of a torus. In particular, \(H\) is always connected and of maximal rank, \ie it contains a maximal torus \(T\) of \(G\). In the special case when \(H=T\), we speak of the \deffont{full flag variety} \(G/T\). We may and will always assume that \(G\) is simply-connected.

The thus defined flag varieties are indeed complex varieties. In fact, they may be realized algebraically as quotients
\[
  G_\C/P,
\]
where \(G_\C\) is the semi-simple algebraic group over \(\C\) corresponding to \(G\) and \(P\) is a parabolic subgroup \cite{Serre:G/P}*{Th^^e9or^^e8me~2}.
Then the Bruhat decomposition of \(G_\C\) yields a cellular decomposition, in the algebraic sense, of \(G_\C/P\).
In particular, flag varieties may be given the structure of CW complexes with cells only in even dimensions. This immediately implies:
\begin{lem}\label{lem:K-structure}
  The \(\K\)-groups of a flag variety \(X\) are given by
  \begin{align*}
    \K^0(X) &= \Z^{\oplus n}\\
    \K^1(X) &= 0
  \end{align*}
 where \(n\) is the number of cells of \(X\).
\end{lem}
Indeed, since the singular cohomology of \(X\) is concentrated in even degrees, the Atiyah-Hirzebruch spectral sequence computing its \(\K\)-theory must collapse. The number of cells of a flag variety \(G/H\) is equal to the index of the Weyl group of \(H\) in the Weyl group of \(G\).

In contrast, the Atiyah-Hirzebruch spectral sequence computing the \(\KO\)-groups of \(X\) does not collapse. Taken by its own, this spectral sequence only yields a description of the free parts of the \(\KO\)-groups (see \autoref{lem:KO-structure} below). However, further structural information may be obtained by studying the relation of real to complex \(\K\)-theory. Let \(\eta\) denote the generator of \(\KO^{-1}(\text{point})=\Z/2\), and let
\[
   \KO^{2i}(X) \xrightarrow{c_i} \K^0(X) \xrightarrow{r_i} \KO^{2i}(X)
\]
denote the realification and complexification maps, composed with appropriate powers of the periodicity isomorphism \( \K^i(X)\cong\K^{i-2}(X)\).  These maps fit into a long exact sequence known as the Bott sequence, which, when \(\K^1(X)=0\), breaks up into \(6\)-term exact sequences of the following form:
\begin{equation*}
   0 \rightarrow
   \KO^{2i-1}(X) \xrightarrow{\eta}
   \KO^{2i-2}(X) \xrightarrow{c_{i-1}}
   \K^0(X)       \xrightarrow{r_i}
   \KO^{2i}(X)   \xrightarrow{\eta}
   \KO^{2i-1}(X) \rightarrow
   0
\end{equation*}
In particular, multiplication by \(\eta^2\) induces an isomorphism from the cokernel of \(r_i\) to the kernel of \(c_{i-1}\), which in turn may be identified with the 2-torsion subgroup of \(\KO^{2i-2}(X)\). Writing \(\W^i(X)\) for the cokernel of \(r_i\) as in the introduction, we obtain the following structural lemma:
\begin{lem}[\cite{Hoggar}*{2.1, 2.2}]\label{lem:KO-structure}
 The \(\KO\)-groups of a CW complex \(X\) with cells only in even dimensions have the following structure:
 \begin{alignat*}{3}
    &\KO^{2i}(X)  &&= \W^{i+1}(X)\cdot\eta^2 \oplus \Z^{\oplus n_{2i}} \\
    &\KO^{2i+1}(X) &&= \W^{i+1}(X)\cdot \eta
  \end{alignat*}
Here, the ranks of the free parts of the groups of even degrees are\\
\mbox{}\hspace{1cm} \(n_0=n_4=\) the number of cells of \(X\) of dimension a multiple of \(4\),\\
\mbox{}\hspace{1cm}  \(n_2=n_6=\) the number of remaining cells of \(X\).
\end{lem}

\begin{rem}[Mulitplicative structure]\label{rem:KO-ring-structure}
It follows from \autoref{lem:KO-structure} that the products \( \KO^{\text{odd}}(X)\otimes\KO^{\text{odd}}(X) \rightarrow \KO^{\text{even}}(X) \) are completely described by the products in \(\W^*(X)\), in view of the following commutative diagram:
\[\xymatrix{
{\W^i(X)\otimes\W^j(X)}          \ar[r]\ar[d]^{\eta\otimes\eta}_{\cong} & {\W^{i+j}(X)}     \ar@{_{(}->}[d]^{\eta^2} \\
{\KO^{2i-1}(X)\otimes\KO^{2j-1}(X)} \ar[r]                          &  {\KO^{2i+2j-2}(X)}
}
\]
\end{rem}

\subsection{Witt rings are Tate cohomology rings}\label{sec:Witt-is-Tate}
We now describe how the Witt ring \(\W^*(X)\) of a flag variety may be obtained directly from the \(\K\)-ring \(\K^0(X)\). In general, the \(\K\)-ring of a space \(X\) comes equipped with an involution
\[
   \K^0(X) \xrightarrow{\star} \K^0(X)
\]
which sends the class of a vector bundle \(\vb E\) to the class of the dual bundle \(\vb E^\dual\). This involution allows us to view \(\K^0(X)\) as a \(\Z/2\)-module. Let \(h^i(X)\) denote the Tate cohomology groups of \(\Z/2\) with coefficients in \(\K^0(X)\). These groups are 2-periodic in \(i\), and we will write \(h^+(-)\) and \(h^-(-)\) to denote the groups of even and odd degrees, respectively. Concretely, these groups may be defined as follows:
\begin{align*}
  h^+(X) &= \frac{\ker(\id-\;\star)}{\im(\id+\;\star)}\\
  h^-(X) &= \frac{\ker(\id+\;\star)}{\im(\id-\;\star)}
\end{align*}
The realification and complexification maps descend to the following well-defined maps between the Tate cohomology groups and quotients/subgroups of the \(\KO\)-groups of \(X\):\footnote{%
We use the notation \(f\under A\) and \(B/f\) to denote the kernel and cokernel of a morphism \(f\colon A\rightarrow B\), respectively.}
\begin{equation*}
  \KO^{2i}(X)/r \xrightarrow{\bar c_i}
  h^i(X)       \xrightarrow{\bar r_i}
  c\under\KO^{2i}(X)
\end{equation*}
The cokernels \(\KO^{2i}(X)/r\) are of course precisely the Witt groups \(\W^i(X)\).
A crucial observation is that these maps even allow us to identify the Witt groups of \(X\) with its Tate cohomology groups, at least up to a slight difference in grading.
\begin{Bousfield-lem}
For any space \(X\) with \(\K^1(X)=0\), the complexification and realification maps induce isomorphisms
\begin{align*}
\W^0(X) \oplus \W^2(X) \xrightarrow[\mm{\bar c_0&\bar c_2}]{\;\cong\;}
   h^+(X) &\xrightarrow[\mm{\bar r_{-1} \\\bar r_1}]{\;\cong\;}
   \;c\under\!\KO^{-2} \oplus \;c\under\!\KO^2(X) \\
\W^1(X) \oplus \W^3(X) \xrightarrow[\mm{\bar c_1&\bar c_3}]{\;\cong\;}
   h^-(X) &\xrightarrow[\mm{\bar r_0\\\bar r_2}]{\;\cong\;}
   \;\;c\under\!\KO^0\;\;  \oplus \;c\under\!\KO^4(X)
\end{align*}
The composition along each row is given by \( \mm{\eta^2 & 0 \\ 0 & \eta^2 } \).
\end{Bousfield-lem}
\begin{proof}
This is essentially \cite{Bousfield:2-primary}*{Lemma~4.7}.  The realification and complexification maps and the involution satisfy the following relations \cite{Bousfield:K-local-spectra}*{4.7}:
\begin{align*}
  c_i r_i &= \id + (-1)^i\star  & \star c_i &= (-1)^i c_i  &  r_{i-1} c_i &= \eta^2\\
  r_i c_i &= 2                &  r_i  \star &= (-1)^i r_i&
\end{align*}
The claim of the lemma concerning the composition along each row follows immediately. Moreover, we have already seen in \autoref{lem:KO-structure} that multiplication by \(\eta^2\) induces an isomorphism from the Witt groups of \(X\) to the kernels of the complexification maps, so the first map in each row must be injective while the second map must be surjective.

We now show that injectivity of \( (\bar c_{i-2} \; \bar c_i ) \) implies injectivity of \(\leftexp{t}{(\bar r_{i-1} \; \bar r_{i+1})} \):
Suppose \(\bar x\in h^i(X) \) is contained in the kernel of \(\leftexp{t}{(\bar r_{i-1} \; \bar r_{i+1})} \).
Choose a representative \(x\in\K^0(X)\) such that \(x^* = (-1)^ix\).
Then \(r_{i\pm 1}(x) = 0\) and Bott's sequence implies that for \(j=i-2\) and \(j=i\) there are elements \(y_j\in\KO^{2j}(X)\) such that \(c_j(y_j) = x \).
Consequently, the image of  \((\bar y_{i-2},\bar y_i) \) under \((\bar c_{i-2} \; \bar c_i) \) is zero in the two-torsion group \(h^+(X)\).
So the injectivity of the latter map implies that \(\bar y_{i-2}\) and \(\bar y_i\) vanish in \(\W^{i-2}(X)\) and \(\W^i(X)\), respectively.
In particular, \(y_i\) is contained in the image of \(r_i\) and, thus, \(x\) is contained in the image of \(c_ir_i=\id+(-1)^i\star\). We deduce that \(\bar x = 0\) in \(h^i(X)\), as required.
\end{proof}

The ring structure of \(\K^0(X)\) induces a ring structure on the direct sum of the Tate cohomology groups
\[
   h^*(X) := h^+(X)\oplus h^-(X),
\]
so we will call \(h^*(X)\) the Tate cohomology ring of \(X\). Since complexification is a ring homomorphism \(\KO^*(X)\rightarrow\K^*(X)\), the first two isomorphisms in Bousfield's Lemma even induce an isomorphism of \emph{rings}
\begin{equation}\label{eq:Bousfield}
  \bar c \colon \W^*(X) \xrightarrow{\;\cong\;} h^*(X)
\end{equation}
mapping the sum of the Witt groups of even degrees to \(h^+(X) \) and the sum of the Witt groups of odd degrees to \(h^-(X)\).
This ring isomorphism will be our main tool in the computation of the Witt rings of full flag varieties.
The isomorphism induced by realification will be used only to identify the \(\Z/4\)-grading of \(\W^*(X)\).

\begin{rem}\label{rem:Witt}%
An analogue of Bousfield's Lemma holds for Balmer's algebraic Witt ring of any smooth cellular variety over an algebraically closed field of characteristic not two \cite{Me:Vanishing}*{Theorem~2.3}.  As a corollary, we obtain a new proof of the fact that the algebraic Witt groups of complex flag varieties agree with the topological Witt groups considered here. This approach is significantly more elementary than the one taken in \cite{Me:WCCV}.  Moreover, it is immediate that we have an isomorphism of Witt \emph{rings}.
\end{rem}

\section{Representations and homogeneous bundles}
As we have seen, the \(\K\)-group \(\K^0(G/H)\) of a flag variety is rather simple. However, a theorem of Hodgkin asserts that it has an interesting ring structure which may be described in terms of representations of \(H\) and \(G\):
\[
    \K^0(G/H)\cong\factor{\RepC(H)}{\keridealC{G}},
\]
where \(\RepC(H)\) is the complex representation ring of \(H\) and \(\keridealC{G}\) is the ideal generated by rank zero virtual representations of \(G\). The theorem  may be divided into two separate assertions:
\begin{enumerate}[label=(\alph*)]
\item We have an epimorphism \(\aC\colon \RepC(H)\twoheadrightarrow \K^0(G/H)\).
\item The kernel of this epimorphism is given by \(\keridealC{G}\).
\end{enumerate}
In \autoref{sec:homogeneous-bundles}, we briefly explain why the first assertion for real representations and vector bundles fails. The second assertion is briefly discussed in \autoref{sec:Hodgkin}. Apart from Hodgkin's Theorem itself, the only part of this discussion that will be relevant to our computations is the \emph{existence} of a real analogue of \(\aC\).
First, however, we need to fix some notation regarding representation rings.

\subsection{Representation rings of compact Lie groups}\label{sec:reps}
The complex representation ring of a compact Lie group \(G\) will be denoted \(\RepC(G)\). Similarly, the \(\K\)-groups of the categories of real and quaternionic representations of \(G\) will be denoted \( \RepR(G) \) and \( \RepH(G) \), respectively. While \( \RepR(G) \) again has a natural ring structure, \(\RepH(G)\) does not: the tensor product of two quaternionic representations is not quaternionic but real \cite{Adams:Lie}*{Definition~3.7}. It is therefore often convenient to consider real and quaternionic representations simultaneously, \ie to consider the \(\Z/2\)-graded ring \(\RepR(G)\oplus\RepH(G)\).

Just as for vector bundles, we have dualization, realification/quaternionification and complexification maps:\footnote{%
We deviate from the traditional notations \(c'\) and \(q\) for  \(c_2\) and \(r_2\) to emphasize the analogy with K-theory.}
\begin{gather*}
    \RepC(G)\xrightarrow{\star} \RepC(G)\\
\begin{aligned}
    \RepR(G) \xrightarrow{c_0}  &\;\RepC(G) \xrightarrow{r_0} \RepR(G) \\
    \RepH(G) \xrightarrow{c_2} &\;\RepC(G) \xrightarrow{r_2} \RepH(G)
\end{aligned}
\end{gather*}
Moreover, for any complex representation \( \lambda \) of \( G \), the tensor product \( \lambda\lambda^* \) carries a canonical real structure \cite{Adams:Lie}*{Lemma~7.2}. This induces a  map
\[
   \RepC(G) \xrightarrow{\phi} \RepR(G)
\]
such that \(c\phi(x) = xx^*\).

The complexification maps \(c_0\) and \(c_2\) are injective for any compact Lie group \(G\) \cite{Adams:Lie}*{Proposition~3.27}. An irreducible complex representation of \(G\) is said to be of \deffont{real type} if it is contained in the image of \(c_0\), and it is said to be of \deffont{quaternionic type} if it is contained in the image of \(c_2\).

\subsubsection{\ensuremath{G}-equivariant \ensuremath{\K}-theory}
Representations may be viewed as equivariant vector bundles over a point, so we can identify  \(\RepC(G)\) and \(\RepR(G)\) with the \(G\)-equivariant \(\K\)-groups \(\K^0_G(\point)\) and \(\KO^0_G(\point)\). We use equivariant \(\K\)-theory to define \(\Z/2\)- and \(\Z/8\)-graded complex and real representation rings as follows:
\begin{alignat*}{3}
   &\RepC^i(G)&&:=\K^i_G(\point) \\
   &\RepR^i(G)&&:=\KO^i_G(\point)
\end{alignat*}
The values of these groups in each degree may be expressed purely in terms of \(\RepC(G)\), \(\RepR(G)\) and \(\RepH(G)\) \cite{BrunerGreenlees:ko}*{Theorem~2.2.12}:
\begin{alignat*}{5}
      &\RepC^0(G) = \RepC(G)
&     &\RepC^1(G) = 0
\\\rule{0pt}{12pt}
      &\RepR^0(G) = \RepR(G)
&\quad&\RepR^{-1}(G) = \factor{\RepR(G)}{r_0}    &\;\;&(\cong W^0(G))
\\    &\RepR^2(G) = \factor{\RepC(G)}{\RepR(G)}
&     &\RepR^1(G) = 0                           &&(= W^1(G))
\\    &\RepR^4(G) = \RepH(G)
&     &\RepR^3(G) = \factor{\RepH(G)}{r_2}      &&(\cong W^2(G))
\\    &\RepR^6(G) = \factor{\RepC(G)}{\RepH(G)}
&     &\RepR^5(G) = 0                           &&(=W^3(G))
\end{alignat*}
The maps \(\eta\colon\RepR^{2i}(G)\twoheadrightarrow\RepR^{2i-1}(G)\) are the obvious epimorphisms, while the maps \(\eta\colon\RepR^{2i+1}(G)\hookrightarrow\RepR^{2i}(G)\) are monomorphisms induced by complexification.   Define the Witt groups of \(G\) by \(\W^i(G):=\RepR^{2i}(G)/r_i\) and write \(h^\pm(R(G))\) for the Tate cohomology of \(R(G)\) with respect to the involution \(\star\).  Then Bousfield's Lemma holds with the same proof as in \autoref{sec:Witt-is-Tate}:
\begin{equation}\label{eq:Bousfield-for-Rep}
\begin{aligned}
  h^+(\RepC(G))&\cong\W^{0,2}(G)\\
  h^-(\RepC(G))&\cong 0
\end{aligned}
\end{equation}

\subsubsection{The case of a torus}
The complex representation ring of a torus \(T\) may be identified with the ring of Laurent polynomials in \(\dim T \) variables:
\[
   \RepC(T) = \Z[x_1^{\pm 1}, \cdots, x_{\dim T}^{\pm 1}]
\]
The Tate cohomology of this ring is easy to compute: dualization corresponds to the involution sending \( x_i \) to \( x_i^{-1} \), and this implies that \( h^*(\RepC(T)) \) is trivial:
\begin{equation}\label{eq:Tate-of-T}
   h^*(\RepC(T)) = \Z/2\cdot\bar 1
\end{equation}
This triviality is the main simplification in the computation of the Witt rings of full flag varieties \(G/T\) as opposed to Witt rings of arbitrary flag varieties.

\begin{rem}\label{rem:RepR-T}
If follows from the triviality of \(h^*(\RepC(T))\) that the real representation ring \(\RepR(T)\) may be identified with the subring of \(\RepC(T)\) fixed by the involution.
However, an explicit description of this subring in terms of generators and relations is fairly complicated: see \cite{AsteyEtAl}*{Theorem~3.2}.
\end{rem}

\subsubsection{The simply-connected case}
In general, for any compact Lie group \(G\) with maximal torus \(T\), the complex representation ring \(\RepC(G)\) may be identified with a subring of \(\RepC(T)\) via the restriction map --- it is the subring fixed by the action of the Weyl group. When \(G\) is simply-connected, one may deduce that \(\RepC(G)\) is a polynomial ring on certain irreducible representations known as the \deffont{basic representations} of \(G\) \cite{Adams:Lie}*{Theorem~6.4.1}. Moreover, the set \(B(G)\) of these basic representations may be partitioned into disjoint subsets
\[
   B(G) = B_\C \cup c_0(B_\R) \cup c_2(B_\H)
\]
consisting of the basic representations of complex, real and quaternionic types, respectively (\cf \cite{Davis:Rep}).
The numbers \(b_\C(G)\), \(b_\R(G)\) and \(b_\H(G)\) appearing in \autoref{thm:main} are the sizes of these subsets.

While representations of real and quaternionic types are self-dual, the basic representations of complex type arise as mutually dual pairs. We may thus choose a subset \( B_\C'\subset B_\C\) such that we obtain a refined decomposition
\[
   B(G) = B_\C' \cup (B_\C')^* \cup c_0(B_\R) \cup c_2(B_\H).
\]
Finally, it will often be convenient to replace the given generators  by generators of rank zero. To indicate this, we will write \( \reduced \zeta \) for the reduced virtual representation associated with a representation \( \zeta \), \ie we define \( \reduced \zeta:= \zeta - \rank \zeta \). Thus, we may write the representation ring as
\begin{equation*}\label{eq:Rep-of-G}
   \RepC(G) = \Z[\reduced\lambda, \reduced\lambda^*,c_0(\reduced \mu),c_2(\reduced \nu)]_{\lambda\in B_\C', \mu\in B_\R, \nu\in B_\H}.
\end{equation*}

\subsection{Homogeneous bundles}\label{sec:homogeneous-bundles}
In \cite{AtiyahHirzebruch:VectorBundles}*{Conjecture~5.7}, Atiyah and Hirzebruch conjectured that all stable isomorphism classes of complex vector bundles on a flag variety \(G/H\) arise from representations of \(H\), in the following way: Given a complex representation \(\zeta\) of \(H\) of rank \( r \), we may consider \(G\times \C^r\) as an \( H \)-space, with \( H \) acting on \( G \) via multiplication and on \( \C^r \) via \( \zeta\). Then the quotient space \( (G\times \C^r)/H \) is a \(G\)-equivariant complex vector bundle over \( G/H \). In fact, this construction induces a ring isomorphism
\begin{equation}\label{eq:Segal}
  \phantom{\aC\colon} \RepC(H) \xrightarrow{\cong} \K_G^0(G/H)
\end{equation}
\cite{Segal:EquivariantK}*{\S1}. Forgetting the equivariant structure, we obtain a ring homomorphism
\begin{equation}
  \aC\colon \RepC(H) \xrightarrow{\phantom{\cong}} \K^0(G/H).
\end{equation}
Atiyah and Hirzebruch checked by direct computation that this homomorphism is surjective in many cases \cite{AtiyahHirzebruch:VectorBundles}*{Theorem~5.8}; their conjecture asserted that it is so in general. The conjecture was verified by Hodgkin and Pittie \citelist{\cite{Hodgkin}*{Lemma~9.2}\cite{Pittie}*{Theorem~3}}.

Of course, the \emph{construction} of \(\aC\) works equally well with real representations and vector bundles. Moreover, the map can be extended to higher degrees: if we interpret \(\RepC(H)\) as \(\K^0_H(\point)\), we may regard \eqref{eq:Segal} as an incarnation of a general isomorphism for equivariant cohomology theories \cite{May:Alaska}*{(4.4) in XVI\S4}. For equivariant \(\KO\)-theory \citelist{\cite{May:Alaska}*{XIV\S4}\cite{BrunerGreenlees:ko}*{2.2}}, we thus have an isomorphism\begin{equation}
  \phantom{\aO^*\colon} \RepR^*(H) \xrightarrow{\cong} \KO_G^*(G/H)
\end{equation}
which, composed with the forgetful map \(\KO_G^*(G/H)\rightarrow\KO^*(G/H)\), yields a ring homomorphism
\begin{equation}
   \aO^*\colon \RepR^*(H) \xrightarrow{\phantom{\cong}} \KO^*(G/H)
\end{equation}
analogous to the morphism \(\aC\) above.

However, \(\aO^*\) is not generally surjective, even in degree zero. We may easily pin down its cokernel in terms of Witt groups.
\begin{prop}\label{prop:coker-aO}
The cokernel of
\[
   \RepR^{2*}(H) \xrightarrow{\aO^{2*}} \KO^{2*}(G/H) \\
\]
coincides with the cokernel of the induced map \(\W^*(H) \rightarrow \W^*(G/H)\).
\end{prop}
\begin{example} Consider a full flag variety \(G/T\). By equations \eqref{eq:Bousfield-for-Rep} and \eqref{eq:Tate-of-T} of \autoref{sec:reps} we have \( \W^*(T)= \Z/2 \), and the map induced by \(\aO\) is simply the inclusion of the trivial subring \(\Z/2 \hookrightarrow \W^*(G/T)\). Our computation of the Witt ring of \(G/T\) in \autoref{thm:main} will show in particular that the cokernel of this inclusion is far from being trivial. For example, \autoref{table:bbb} tells us that \(\W^*(\mathrm{SU}(6)/T)\) is an exterior algebra on one generator of degree one and two generators of degree three, so said cokernel is non-trivial in all degrees.
\end{example}
\begin{proof}[Proof of \autoref{prop:coker-aO}]
It follows from the construction of \(\aC\) and \(\aO\) that these maps are compatible with the structure maps \(r_i\) and \(c_i\). In particular, we may consider the following commutative diagram:
\[\xymatrix{
  {\RepC^0(H)} \ar@{->>}[d]^{\aC} \ar[r]^{r_i}
& {\RepR^{2i}(H)}      \ar[d]^{\aR^{2i}}   \ar@{->>}[r]
& {\W^i(H)}  \ar[d]\\
  {\K^0(G/H)}  \ar[r]^{r_i}
& {\KO^{2i}(G/H)} \ar@{->>}[r]
& {\W^i(G/H)}
}\]
This diagram has exact rows by our definition of the Witt groups, and \(\aC\) is surjective by the work of Hodgkin and Pittie mentioned above. This implies that the cokernel of the second vertical map agrees with the cokernel of the third, as claimed.
\end{proof}
The maps \(\aO^{2i-1}\colon \RepR^{2i-1}(H) \rightarrow \KO^{2i-1}(G/H)\) in odd degrees may of course be identified directly with the maps \(\W^i(H)\rightarrow\W^i(G/H)\).

\subsection{Hodgkin's Theorem}\label{sec:Hodgkin}\label{sec:K-and-Rep}
At the beginning of the previous section, we explained how \(\aC\) associates vector bundles over \(G/H\) with representations \(\zeta\) of \(H\).
It follows from the construction that when \(\zeta\) is the restriction of a representation of \(G\), the associated vector bundle may be trivialized. In particular, the ideal
\[
   \keridealC{G}\subset\RepC(H)
\]
generated by virtual representations of \(G\) of rank zero is contained in the kernel of \(\aC\).
In \cite{Hodgkin}*{Lemma~9.2}, Hodgkin shows not only that \(\aC\) is surjective but also that its kernel is precisely this ideal:
\begin{Hodgkin-thm}
For any flag variety \(G/H\) with \(G\) simply-connected, \(\aC\) induces a ring isomorphism
\begin{equation}
   \baraC\colon \factor{\RepC(H)}{\keridealC{G}} \xrightarrow{\cong} \K^0(G/H).
\end{equation}
\end{Hodgkin-thm}
Hodgkin's proof is based on the construction of a K^^fcnneth spectral sequence for equivariant \(\K\)-theory. A different proof, leading to a more general algebraic version of the theorem, is due to Panin \cite{Panin:TwistedFlags}. A key ingredient in both approaches which we will need later is the following theorem of Pittie and Steinberg:
\begin{Steinberg-thm}[\cite{Steinberg:Pittie}*{Theorem~1.1}]\label{thm:Steinberg}
The representation ring \(\RepC(H)\) of a connected subgroup \(H\) of maximal rank in a simply-connected compact Lie group \(G\) is a free \(\RepC(G)\)-module of finite rank.
\end{Steinberg-thm}
Steinberg's proof also shows that the rank of \(\RepC(H)\) over \(\RepC(G)\) is equal to the index of the Weyl group of \(H\) in the Weyl group of \(G\). In particular, the additive description of \( \K^0(G/H) \) given in \autoref{lem:K-structure} may easily be recovered.

We now define a factorization of \(\aO^*\) analogous to the factorization \(\baraC\) of \(\aC\). Let \(\tilde\RepR{}^*(G)\) denote the kernel of the forgetful map \(\RepR^*(G)\rightarrow\RepR^*(\{1\})\), and let
\[
    \keridealO^*(G) \subset \RepR^*(H)
\]
be the ideal generated by the image of \(\tilde\RepR{}^*(G)\) and by the images of \(\keridealC{G}\) under the maps \(r_i\). Then \(\aO^*\) factors through
\begin{equation}
    \baraO^*\colon \factor{\RepR^*(H)}{\keridealO^*(G)} \rightarrow \KO^*(G/H),
\end{equation}
and the factorization is compatible with the structure maps \(r_i\) and \(c_i\). In particular, the following diagram commutes:
\begin{equation}\label{eq:ra-is-ar}
\begin{aligned}
\xymatrix{
  {\factor{\RepC^0(H)}{\keridealC{G}}}\ar[d]^{r_i}\ar[r]^-{\baraC}
& {\K^0(G/H)}               \ar[d]^{r_i}\\
  {\factor{\RepR^{2i}(H)}{\keridealO^{2i}(G)}}\ar[r]^-{\baraO^{2i}}
& {\KO^{2i}(G/H)}\\
}
\end{aligned}
\end{equation}
This will be used to identify the \(\Z/4\)-grading of \(\W^*(G/H)\) in the proof of \autoref{thm:main}.
\begin{rem}[Description of \(\keridealRH{G}\)]
The ideal \(\keridealRH{G}\subset \RepRH(H)\) is generated by real and quaternionic virtual representations of \(G\) of rank zero, together with the images of \(\keridealC{G}\) under realification and quaternionification.
In the notation of \autoref{sec:reps}, \(\keridealC{G}\) and \(\keridealRH{G}\) may be written as follows:
\begin{align*}
\keridealC{G}
  &= (\reduced\lambda,\reduced\lambda^*, c_0\reduced\mu,c_2\reduced\nu)_{\lambda,\mu,\nu} \\
\keridealRH{G}
  &= (\phi\reduced\lambda,\reduced\mu,\reduced\nu)_{\lambda,\mu,\nu}
     + r_0(\keridealC{G})
     + r_2(\keridealC{G})
\end{align*}
\end{rem}
\begin{rem}[Injectivity of \(\baraRH\)]\label{rem:ker-alphaRH}
It seems natural to ask whether the ideal \(\keridealRH{G}\) coincides with the kernel of
\(
   \aRH
\).
Our computations for flag varieties in \autoref{sec:Tate-of-quotients} imply that this is true only in special cases:
\end{rem}
\begin{prop}
For a complete flag variety \(G/T\), the map
\[
   \baraO^{0,4}\colon \factor{\RepR^{0,4}(T)}{\keridealO^{0,4}(G)} \rightarrow \KO^{0,4}(G/T)
\]
is injective only when \(G\) is of one of the following low-dimensional types:
\[
A_n \text{ with } 1\leq n\leq 4, \quad A_n\times A_m \text{ with } 1\leq n,m\leq 2, \quad B_2,\quad G_2.
\]
\end{prop}
\begin{proof}
Since this proposition is not central to the aims of this paper, we include only a brief sketch of the argument. For any compact Lie group \(H\), and for any space \(X\), the pairs of abelian groups \((\RepC(H),\RepRH(H))\) and \((\K^0(X),\KO^{0,4}(X))\) fit into complexes of the form
\begin{smalldiagram}
\xymatrix@C=2em@R=2em{
  {\dots}    \ar[r]
& {\RepC(H)} \ar[r]^{1-\star}
& {\RepC(H)} \ar[r]^{\mm{r_0\\r_2}}
& {\RepRH(H)}\ar[r]^{\mm{c_0&-c_2}}
& {\RepC(H)} \ar[r]^{1-\star}
& {\RepC(H)} \ar[r]
& {\dots}
\\
  {\dots}       \ar[r]
& {\K^0(X)}     \ar[r]^{1-\star}
& {\K^0(X)}     \ar[r]^{\mm{r_0\\r_2}}
& {\KO^{0,4}(X)}\ar[r]^{\mm{c_0&-c_2}}
& {\K^0(X)}     \ar[r]^{1-\star}
& {\K^0(X)}     \ar[r]
& {\dots}
}
\end{smalldiagram}%
The relevant Bott exact sequences imply that these complexes are exact \cite{Bousfield:2-primary}*{Theorem~4.4}.
Taking \(X\) to be a flag variety \(G/H\), we may compare the two sequences via the maps \(\aC\) and \(\aRH\).
Basic homological algebra then shows that the kernel of \(\baraRH\) may be identified with the homology of the corresponding complex for \((\keridealC{G},\keridealRH{G})\) at \(\keridealRH{G}\to\keridealC{G}\to \keridealC{G}\), which we may view as the cokernel of
\begin{align*}
\factor{\keridealR{G}\oplus\keridealH{G}}{\mm{r_0\\r_2}} &\xrightarrow{\mm{ c_0 & -c_2 }} \ker(\keridealC{G}\xrightarrow{1-\star}\keridealC{G}).\\
\intertext{%
This cokernel maps isomorphically to the cokernel of
}
  \factor{\keridealR{G}}{r_0} \oplus \factor{\keridealH{G}}{r_2} &\xrightarrow{\mm{\bar c_0&\bar c_2}} h^+(\keridealC{G}).
\end{align*}
In particular, \(\baraO^{0,4}\) is injective if and only if \((\bar c_0\;\;\bar c_2)\) is surjective.
For a full flag variety \(G/T\), the domain of \((\bar c_0\;\;\bar c_2)\) is easy to compute: the fact that
\[
  \RepR^0(T)/r_0\oplus\RepR^4(T)/r_2 = \W^{0,2}(T) =  \Z/2\cdot \bar 1_\R
\]
implies that \(\keridealR{G}/r_0\oplus\keridealH{G}/r_2\) is generated as a \(\Z/2\)-module by the elements \(\phi\reduced\lambda\), \(\reduced\mu\) and \(\reduced\nu\)  with \(\lambda\), \(\mu\) and \(\nu\) in \(B'_\C(G)\), \(B_\R(G)\) and \(B_\H(G)\) respectively.

On the other hand, the computations in \autoref{sec:Tate-of-quotients} will imply that \(h^+(\keridealC{G})\) is additively isomorphic to the odd-degree part of an exterior algebra on \(\frac{b_\C}{2} + b_\R + b_\H\) generators of degree \(1\):
\[
   h^+(\keridealC{G}) \cong h^-(\RepC(T)/\keridealC{G}) = \left[\Lambda_{\Z/2}(u_\lambda,v_\mu,w_\nu)_{\lambda,\mu,\nu}\right]^{\odd}
\]
Moreover, one may check that under this additive isomorphism the generators \(u_\lambda\), \(v_\mu\) and \(w_\nu\) correspond to the elements \(c_0\phi\reduced\lambda\), \(c_0\reduced\mu\) and \(c_2\reduced\nu\) in \(h^+(\keridealC{G})\), respectively.
The map  \((\bar c_0\;\;\bar c_2)\) is simply the inclusion of the \(\Z/2\)-linear subspace spanned by these elements. It is surjective if and only if the exterior algebra \(\Lambda_{\Z/2}(u_\lambda,v_\mu,w_\nu)_{\lambda,\mu,\nu}\)  is generated by at most \(2\) elements, \ie if and only if
\[
   \frac{b_\C(G)}{2}+b_\R(G)+b_\H(G) \leq 2.
\]
This holds only in the cases listed in the proposition.
\end{proof}

\newpage
\section{Witt rings of full flag varieties}\label{sec:W-of-G/T}
Hodgkin's Theorem (\autoref{sec:Hodgkin}) and Bousfield's Lemma (\autoref{sec:Witt-is-Tate}) have the following corollary:
\begin{cor}\label{cor:main}
For any flag variety \(G/H\) with \(G\) simply-connected, we have an isomorphism of \(\Z/2\)-graded rings
\[
   \W^*(G/H) \xrightarrow{\cong} h^*\left(\factor{\RepC(H)}{\keridealC{G}}\right),
\]
where \(h^*(-):=h^+(-)\oplus h^-(-)\) denotes Tate cohomology.
\end{cor}
\begin{proof}
It suffices to observe that the isomorphism \(\baraC\colon \factor{\RepC(H)}{\keridealC{G}} \cong \K^0(G/H)\) of Hodgkin's Theorem is an isomorphism of rings with involution, for the involutions induced by dualizing representations and vector bundles. We thus obtain an isomorphism of Tate cohomology rings, which, combined with Bousfield's Lemma \eqref{eq:Bousfield}, yields the desired identification:
\begin{equation}\label{eq:main-iso}
   \W^*(G/H)\xrightarrow[\bar c]{\cong} h^*(G/H) \xleftarrow[\baraC]{\cong} h^*\left(\factor{\RepC(H)}{\keridealC{G}}\right)
\end{equation}
Note that \(h^*(G/H)\) is our notation for \(h^*(\K^0(G/H))\).
\end{proof}

We now apply this corollary to a full flag variety \(G/T\). Recall from \autoref{sec:reps} that \( h^*(\RepC(T)) \) is trivial:
\begin{align*}
   h^+(\RepC(T)) &= \Z/2\cdot\bar 1 \\
   h^-(\RepC(T)) &= 0
\end{align*}
On the other hand, we have seen that the complex representation ring of \(G\) may be written as
\[
   \RepC(G) = \Z[\reduced\lambda, \reduced\lambda^*,c_0(\reduced \mu),c_2(\reduced \nu)]_{\lambda\in B_\C', \mu\in B_\R, \nu\in B_\H}.
\]
Viewing \(\RepC(G)\) as a subring of \(\RepC(T)\) via the restriction map, we may consider the following virtual representations of \(G\) as elements of \(\RepC(T)\):
\begin{align*}{}
  \reduced\lambda\reduced\lambda^* \quad & \text{ with } \lambda\in B_\C' \\
  c_0(\reduced \mu)                  \quad &  \text{ with } \mu\in B_\R\\
  c_2(\reduced \nu)                 \quad & \text{ with } \nu\in B_\H
\end{align*}
All of these elements are self-dual, so they define classes in \(h^+(\RepC(T))\). But since \(h^+(\RepC(T))\) is so simple and the elements have rank zero, the corresponding classes must vanish. By definition of \(h^+\), this means that we may find elements \(u_\lambda\), \(v_\mu\) and \(w_\nu\) in \( \RepC(T) \) such that
\begin{equation}\label{eq:uvw-relations}
\left.
\begin{aligned}
  u_\lambda + u_\lambda^* &= \reduced\lambda\reduced\lambda^*  \\
  v_\mu + v_\mu^* &= c_0(\reduced \mu)                         \\
  w_\nu + w_\nu^* &= c_2(\reduced \nu)
\end{aligned}
\right\}\text{ in } \RepC(T).
\end{equation}
Consider these equations in \(\K^0(G/T) \cong \RepC(T)/\keridealC{G}\). Here, the right-hand sides vanish since they are given by virtual representations of \(G\) of rank zero. Thus, the images of \(u_\lambda\), \(v_\mu\) and \(w_\nu\) in \( \K^0(G/T) \) are anti-self-dual. With a little more work one may see that these images generate the Tate cohomology ring of \(G/T\). In fact, we have the following:

\begin{prop}\label{prop:h(G/T)}
Let \( G/T \) be a full flag variety with \(G\) simply-connected. The Tate cohomology ring
\[
   h^*(\factor{\RepC(T)}{\keridealC{G}})
\]
is an exterior \(\Z/2\)-algebra on certain generators
   \(u_\lambda, v_\mu, w_\nu \in h^-(-)\).
These generators may be represented by elements of \(\RepC(T)\) satisfying \eqref{eq:uvw-relations} above.
\end{prop}
\begin{proof}[Proof, assuming \autoref{cor:h-summary}]
We defer the main part of the argument to \autoref{sec:Tate-of-quotients}.
Here, we simply observe that the claim follows directly from \autoref{cor:h-summary} with \( R:=\RepC(G) \) and \( A:=\RepC(T) \).
Indeed, all assumptions needed have already been checked:
\(R\) is a polynomial ring on pairs of mutually dual generators and certain self-dual generators,
\(A\) is a free \(R\)-module of finite rank by Steinberg's Theorem (\autoref{thm:Steinberg}),
and \(h^*(A)\) is trivial.
\end{proof}

\begin{thm}\label{thm:main}
 Let \( G/T\) be a flag variety with \(G\) simply-connected. The total Witt ring of \( G/T \) is an exterior \(\Z/2\)-algebra on certain generators
\begin{alignat*}{3}
   &c_\nu             &&\in \W^1(G/T) \\
   &a_\lambda, b_\mu    &&\in \W^3(G/T)
\end{alignat*}
indexed by the basic representations \(\lambda\in B'_\C\), \(\mu\in B_\R\) and \(\nu\in B_\H\):
\[
    \W^*(G/T) \cong \Lambda(a_\lambda,b_\mu, c_\nu)_{\lambda\in B'_\C, \mu\in B_\R, \nu\in B_\H}
\]
\end{thm}
\begin{proof}
Let \(u_\lambda\), \(v_\mu\) and \(w_\nu\) be the generators of \autoref{prop:h(G/T)}, and let \(a_\lambda\), \(b_\mu\) and \(c_\nu\) be their preimages under the isomorphism \eqref{eq:main-iso} (with \(H=T\)), that is, the preimages of the elements \(\baraC u_\lambda\), \(\baraC v_\mu\) and \(\baraC w_\nu\) under \(\bar c\).
Since \(u_\lambda\), \(v_\mu\) and \(w_\nu\) are elements of \(h^-(-)\), we have \(a_\lambda,b_\mu,c_\nu\in\W^{1,3}(G/T)\).

It remains to show that these elements remain homogeneous with respect to the \( \Z/4\)-grading of the total Witt group, of degrees \(3\), \(3\) and \(1\), respectively. To see this, we first observe that the second isomorphism
\[
   \mm{\bar r_0\\\bar r_2}\colon h^-(G/T) \xrightarrow{\cong} c\under\KO^0(G/T) \oplus c\under\KO^4(G/T)
\]
of Bousfield's Lemma preserves the grading of \(h^-(G/T)\) corresponding to the decomposition  into \(\W^1(G/T)\oplus\W^3(G/T)\). It therefore suffices to show that:
\begin{align}\label{eq:mainproof1}
 r_0(\baraC u_\lambda) &=0 \quad\text{in } \KO^0(G/T) \notag \\
 r_0(\baraC v_\mu)    &=0 \quad\text{in } \KO^0(G/T)         \\
 r_2(\baraC w_\nu)    &=0 \quad\text{in } \KO^4(G/T)  \notag
\end{align}

In order to obtain these relations, we note that the injectivity of the complexification maps \(c_0\) and \(c_2\) on representations of \(G\) allows us to rewrite equations \eqref{eq:uvw-relations} as
\begin{alignat*}{4}
  r_0 u_\lambda  &= \phi(\reduced \lambda_i) &\quad&\text{in } \RepR^0(T)\\
  r_0 v_\mu     &= \reduced \mu_i            &\quad&\text{in } \RepR^0(T)\\
  r_2 w_\nu     &= \reduced \nu_i            &\quad&\text{in } \RepR^4(T)
\end{alignat*}
Since all representations on the right-hand sides are restrictions of representations of \(G\),
\begin{equation*}
\left.
\begin{alignedat}{4}
   r_0 u_\lambda&=0  \\
   r_0 v_\mu   &=0  \\
   r_2 w_\nu   &=0
\end{alignedat}
\right\}\text{ in } \factor{\RepRH(T)}{\keridealRH{G}}.
\end{equation*}
We now apply the map \(\baraO\) defined in \autoref{sec:Hodgkin} to these equations. By \eqref{eq:ra-is-ar} this yields the desired relations \eqref{eq:mainproof1}.
\end{proof}

Finally, we indicate how to translate the above computation to the algebraic setting of a full flag variety \(G/B\) over an algebraically closed field of characteristic not two.

\begin{proof}[Proof of {\theoremprime}]
  In the algebraic setting, the role of the topological \(\K\)-group is played by the algebraic \(\K\)-group \(\K_0(G/B)\), while the roles of the \(\KO\)-groups of even degrees \(\KO^{2i}\) are played by Balmer and Walter's Grothendieck-Witt groups \(\GW^i(G/B)\) \citelist{\cite{Walter:TGW}}.  The complexification maps correspond to the forgetful maps \(\GW^i(G/B)\to\K_0(G/B)\), the realification maps correspond to the hyperbolic maps \(\K_0(G/B)\to\GW^i(G/B)\), and the cokernels of the hyperbolic maps are Balmer's Witt groups \(\W^i(G/B)\).  Moreover, over an algebraically closed field, Bousfield's Lemma (\autoref{sec:Witt-is-Tate}) holds as before \cite{Me:Vanishing}*{Theorem~2.3}.

  Similarly, we can consider the algebraic \(\K\)-group \(\K_0(\cat Rep(G))\) and the \GrothendieckWitt and Witt groups \(\GW^i(\cat Rep(G))\) and \(\W^i(\cat Rep(G))\) of the category \(\cat Rep(G)\) of finite-dimensional representations of an affine algebraic group \(G\) in place of the groups \(\RepC(G)\), \(\RepR^{2i}(G)\) and \(\RepR^{2i+1}(G)\).  Then again \(h^*(G) = \W^*(G)\) for an affine algebraic group over an algebraically closed field \cite{Me:Vanishing}*{Theorem~2.2}.

Finally, as already mentioned in \autoref{sec:Hodgkin}, Hodgkin's Theorem also holds in the algebraic setting \cite{Panin:TwistedFlags}: for any simply-connected semi-simple algebraic group \(G\), and for any parabolic subgroup \(P\subset G\), there is a ring isomorphism
\[
\K_0(G/P) \cong \factor{\K_0(\cat Rep(P))}{\ideal a},
\]
where \(\ideal a\subset \K_0(\cat Rep(P))\) is the ideal generated by restrictions of rank zero classes in \(\K_0(\cat Rep(G))\).
This allows us to carry over the topological computations described above.
\end{proof}

\section{Tate cohomology of some quotient rings}\label{sec:Tate-of-quotients}
The aim of this section is to complete our computations by proving \autoref{prop:h(G/T)}. This is accomplished in \autoref{cor:h-summary}.

We first need some terminology.
We define a \deffont{\star-module} to be an abelian group equipped with an involution \(\star\) which is a group isomorphism, a \deffont{\star-ring} to be a commutative unital ring equipped with an involution \(\star\) which is a ring isomorphism, and a \deffont{\star-ideal} in a \star-ring \( {A} \) to  be an ideal preserved by the involution on \(A\).
We will say that a {\star}-ideal \(\ideal a\subset A\) is \deffont{generated} by certain elements \(a_1,\dots,a_n\) and write
\[
  \ideal a = (a_1,\dots,a_n)
\]
if \(\ideal a\) is generated as an ideal in the usual sense by the elements \(a_1,a_1^\star,\dots,a_n,a_n^\star\).

An element \(x\) of a \star-module \(M\) will be called \deffont{self-dual} if \(x=x^*\) and \deffont{anti-self-dual} if \(x^*=-x^*\). The Tate cohomology of \(M\) may be defined as in \autoref{sec:Witt-is-Tate}:
\begin{align*}
   h^+(M)&:=\frac{\ker(\id-\;\star)}{\im(\id+\;\star)}\\
   h^-(M)&:=\frac{\ker(\id+\;\star)}{\im(\id-\;\star)}
\end{align*}
In other words, \(h^+(M)\) is the quotient of the subgroup of self-dual elements by those elements which are ``obviously self-dual'', and similarly \(h^-(M)\) is a quotient of the subgroup of anti-self-dual elements.
For a \star-ring \(A\), the direct sum
\[
   h^*(A):= h^+(A)\oplus h^-(A)
\]
inherits a commutative ring structure, so we will refer to \(h^*(A)\) as the \deffont{Tate cohomology ring} of \(A\). It is easily checked that this ring is \(\Z/2\)-graded, \ie that:
\begin{align*}
h^+(A)\cdot h^+(A)&\subset h^+(A)\\
h^+(A)\cdot h^-(A)&\subset h^-(A)\\
h^-(A)\cdot h^-(A)&\subset h^+(A)
\end{align*}

We begin with a simple Lemma relating the Tate cohomology of principal ideals in a \(\star\)-ring \(A\) to the Tate cohomology of \(A\).
\newcommand{\rlambda}{\reduced \lambda}
\newcommand{\rmu}{\reduced \mu}
\begin{prop}\label{prop:h-principal-ideal}
 Let \( {A} \) be a \star-ring.
 \begin{itemize}
 \item If \( \rmu \in A \) is a self-dual element which is not a zero divisor, then multiplication by \( \rmu \) induces a graded isomorphism
 \[
    h^*(A) \xrightarrow[\cdot \rmu]{\cong} h^*(\rmu A).
 \]
 If, in addition, the class of \( \rmu \) in \( h^+(A)\) is trivial, then
 \[ h^*(A/(\rmu)) \cong h^*(A) \oplus \bar u\cdot h^*(A), \]
 where \( \bar u \) may be represented by any element of \( A \) with the property that \( u+u^* = \rmu \).

 \item If \( (\rlambda, \rlambda^*) \) is a regular sequence in \( A \), then multiplication by the product \( \rlambda\rlambda^* \) induces a graded isomorphism
 \[
    h^*(A) \xrightarrow[\cdot \rlambda\rlambda^*]{\cong} h^*((\rlambda,\rlambda^*) A)
 \]
 If, in addition, the class of \( \rlambda\rlambda^*\) in \( h^+(A) \) is trivial, then
 \[ h^*(A/(\rlambda,\rlambda^*)) \cong h^*(A) \oplus \bar u\cdot h^*(A),\]
 where \( \bar u \) may be represented by any element of \( A \) with the property that \( u+u^* = \rlambda\rlambda^* \).
  \end{itemize}
\end{prop}
\begin{proof}
We prove the second part; the proof of the first is analogous.
Since \( \rlambda\rlambda^* \) is self-dual, the map
 \[
    h^*(A) \xrightarrow{\cdot \rlambda\rlambda^*} h^*((\rlambda,\rlambda^*)A)
 \]
is well-defined.
For injectivity, suppose first that \( \rlambda\rlambda^* x = 0 \) in \( h^+((\rlambda,\rlambda^*)A ) \) for some \( x \in h^+(A) \). Then
\(  \rlambda\rlambda^*x = y + y^*  \) for some \( y \in (\rlambda,\rlambda^*)A \). Write \( y = \rlambda y_1 + \rlambda^* y_2 \) for certain \( y_1, y_2 \in A \). Then
\[
     \rlambda\rlambda^* x = \rlambda(y_1 + y_2^*) + \rlambda^*(y_1^*+y_2).
\]
In particular, \( \rlambda^*(y_1^* + y_2) = 0\) in \( A/(\rlambda) \). By regularity of the sequence \( (\rlambda, \rlambda^*) \), this implies that \(y_1^* + y_2 = 0\) in \( A/(\rlambda) \). So \( y_1^*+y_2 = \rlambda z \) in \( A \), for some \( z \in A \).
Thus,
\[
    \rlambda\rlambda^* x = \rlambda\rlambda^*(z + z^*).
\]
By assumption, \(\rlambda\) is not a zero-divisor in \( A \). Its dual \( \rlambda^* \) cannot be a zero-divisor in \( A \) either. We may therefore conclude that \( x = z+z^* \), which shows that \( x = 0 \) in \( h^+(A) \). The case when \( x \in h^-(A) \) works analogously.

For surjectivity, suppose that \( y \in (\rlambda,\rlambda^*)A \) is an arbitrary self-dual element. Writing \( y = \rlambda y_1 + \rlambda^* y_2\) as in the proof of injectivity, we see that
\[
   \rlambda(y_1-y_2^*) + \rlambda^*(y_2-y_1^*) = 0.
\]
It follows that \( y_2-y_1^* = 0 \) in  \( A/(\rlambda) \), so \( y_2-y_1^* = \rlambda z \) for some \( z \in A \). Thus, we find that \( y = \rlambda y_1 + \rlambda^* y_1^* + \rlambda\rlambda^*z \). So \( y \) is equivalent to \(\rlambda\rlambda^*z\) in \( h^+((\rlambda,\rlambda^*)A) \).  Moreover, \(z\) defines an element in \(h^+(A)\):  the fact that \(y\) is self-dual implies that \(\rlambda\rlambda^* z^* = \rlambda\rlambda^* z\) and thus \(z = z^*\) as neither \(\rlambda\) nor \(\rlambda^*\) is a zero divisor.  Preimages of elements of \( h^-((\rlambda,\rlambda^*)A) \) may be found in an analogous fashion.

Finally, consider the Tate cohomology sequence associated with the short exact sequence of \star-modules
\[
   0 \rightarrow (\rlambda, \rlambda^*)A \rightarrow A \rightarrow A/(\rlambda,\rlambda^*) \rightarrow 0
\]
The assumption on the class of \( \rlambda\rlambda^* \) in \( h^+(A) \) shows that the map induced by the inclusion of \( (\rlambda,\rlambda^*)A \) into \( A \) is zero on cohomology. Thus, the cohomology sequence splits into short exact sequences from which the stated result for \( h^*(A/(\rlambda,\rlambda^*)) \) follows.
\end{proof}

\begin{cor}\label{cor:h-regular-sequence-additive}
Let \( {A} \) be a \star-ring. Let \( \ideal a := (\rlambda_1, \dots, \rlambda_k, \rmu_{k+1}, \dots \rmu_{k+l}) \) be a \star-ideal in \( A \) such that:
\begin{itemize}
\item \( \rlambda_1, \rlambda_1^*, \dots, \rlambda_k, \rlambda_k^*, \rmu_{k+1}, \dots \rmu_{k+l} \) is a regular sequence in \( A \).
\item The generators \( \rmu_{k+1}, \dots, \rmu_{k+l} \) are self-dual.
\item The class of each of the elements \( \rlambda_i\rlambda_i^* \) and \( \rmu_i \) is trivial in \( h^+(A) \).
\end{itemize}
Then we have an isomorphism of \( h^*(A) \)-modules
\[
    h^*(A/\ideal a) \cong h^*(A)\otimes_{\Z/2} \Lambda(\bar u_1,\dots,\bar u_{k+l}),
\]
where the generators \( \bar u_i \) may be represented by elements \( u_i \in A \) such that
  \[
     u_i + u_i^* = \begin{cases}
                   \rlambda_i\rlambda^*_i & \text{for \( i=1,\dots, k \) } \\
                   \rmu_i                & \text{for \( i=k+1,\dots, k+l\) }
                   \end{cases}
  \]
\end{cor}
\begin{proof}
Write \( \ideal a_j\) for the \star-ideal in \( A \) generated by the first \( j \) generators, \ie by \( \rlambda_1,\dots, \rlambda_j \) or  \( \rlambda_1, \dots, \rlambda_k, \dots, \rmu_j \), respectively. Since \( \rlambda_i\rlambda_i^* \) and \( \rmu_i \) are trivial in \( h^+(A) \), they are a fortiori trivial in \( h^+(A/\ideal b) \) for any \star-quotient \( A/\ideal b \) of \( A \). More precisely, if \( u \in A \) has the property that  \( u+u^* = \rlambda_i\rlambda_i^*\) or \( u+u^* = \rmu_i \), then the class of \( u \) in any \star-quotient of \( A \) still has this property. Thus, we may proceed by induction. The case when \( \ideal a \) has zero generators is trivial. To pass from \( j \) to \( j+1 \) generators, we apply  Prop.~\ref{prop:h-principal-ideal} to the classes of \( \rlambda_{j+1},\rlambda_{j+1}^* \) or \( \rmu_{j+1} \) in \( A/\ideal a_j \).
\end{proof}
We now analyse a special case in which we can show that the additive isomorphism in \autoref{cor:h-regular-sequence-additive} is in fact an isomorphism of rings. By an \deffont{augmented \star-ring} we will mean a \star-ring \( A \) together with a surjective \star-morphism \(\rank\colon A \twoheadrightarrow \Z \), where \( \Z \) is equipped with the trivial involution. The image of an element of \( A \) under the augmentation will be referred to as the \deffont{rank} of that element. If \( \ideal a \) is a \star-ideal in \( A \) generated by elements of rank zero, then the quotient \( A/\ideal a \) again has a canonical augmentation. Moreover, an augmentation of \( A \) descends to an augmentation \( \bar\rank\colon h^*(A) \twoheadrightarrow \Z/2 \) mapping elements of \( h^+(A) \) to their rank modulo two and elements of \( h^-(A) \) to zero.
\begin{lem}\label{lem:h-squares}
Let \( {A} \) be an augmented \star-ring with \( h^*(A) \) trivial (\ie such that \( h^+(A) = \Z/2\cdot\bar{1} \) and \( h^-(A) = 0 \)). Let \( \ideal a \) be any \star-ideal in \( A \) generated by elements of rank zero. Then all elements of rank zero in \( h^*(A/\ideal a) \) square to zero.
\end{lem}
\begin{proof}
Take \( \bar{u} \in h^*(A/\ideal{a}) \) such that \( \bar\rank(\bar u) = 0\). Pick a representative \( u \in A \). Then \( uu^* \) defines an element in \( h^+(A) \), which, for rank reasons, must be zero. It follows that, a fortiori, \( \bar{uu^*} = 0 \) in \( h^+(A/\ideal{a}) \). Since \( \bar{u} \) was in \( h^*(A/\ideal{a}) \) to begin with, we have \( \bar{uu^*} = \bar{u}^2 \) in \( h^*(A/\ideal{a}) \).
\end{proof}

\begin{cor}\label{cor:h-regular-sequence-multiplicative}
Let \( {A} \) be an augmented \star-ring containing a \star-ideal \( \ideal a \) generated by elements \(\rlambda_1, \dots, \rlambda_k, \rmu_{k+1}, \dots \rmu_{k+l} \) of rank zero. Assume that:
\begin{itemize}
\item \( \rlambda_1, \rlambda_1^*, \dots, \rlambda_k, \rlambda_k^*, \rmu_{k+1}, \dots \rmu_{k+l} \) is a regular sequence in \(A\).
\item The generators \( \rmu_{k+1}, \dots, \rmu_{k+l} \) are self-dual.
\item  \( h^*(A) \) is trivial.
\end{itemize}
Then we have an isomorphism of \emph{rings}
\[
   h^*(A/\ideal a) \cong \Lambda(\bar u_1, \dots, \bar u_{k+l}).
\]
In particular, \(\bar u_i^2 = 0 \) in \( h^*(A/\ideal a) \) for all \( i \). Representatives of the generators \( \bar u_i \) may be chosen as in \autoref{cor:h-regular-sequence-additive}.
\end{cor}
\begin{proof}
Since \( \rlambda_i\rlambda_i^*\) and \( \rmu_i\) have rank zero, they must be trivial in \( h^+(A) \). We may thus apply \autoref{cor:h-regular-sequence-additive} to obtain the module structure of \( h^*(A) \) and Lemma~\ref{lem:h-squares} to obtain the algebra structure.
\end{proof}

Finally, we specialize to the situation of \autoref{prop:h(G/T)}. Let \( {R} \) be an augmented \star-ring, and let \( {A} \) be an augmented \star-algebra over \({R}\). That is, \(A\) is simultaneously a \star-ring and an augmented algebra over \( {R} \) such that \((r\cdot a)^\star = r^\star a^\star\) for all \(r\in R\) and \(a\in A\). Given an arbitrary element \(\lambda\) in \( R\) (or in \(A\)), we write \(\reduced\lambda\) for the element \(\lambda-\rank(\lambda)\) of rank zero.

\begin{cor}\label{cor:h-summary}
  Let \( {A} \) be an augmented algebra over an augmented \star-ring \( {R} \). Suppose that:
  \begin{itemize}
  \item \( R \) is a polynomial ring of the form \( R= \Z[\lambda_1, \lambda_1^*, \dots, \lambda_k, \lambda_k^*, \mu_{k+1}, \dots \mu_{k+l}]\) with \( \mu_{k+1}, \dots, \mu_{k+l} \) self-dual.
  \item \( A \) is a free \( R \)-module of finite rank.
  \item \( h^*(A) \) is trivial, \ie \( h^+(A) = \Z/2\cdot\bar{1} \) and \( h^-(A) = 0 \).
  \end{itemize}
  Let \(\ideal a\subset A\) be the ideal generated by all elements of \(R\) of rank zero.
  Then
  \[
     h^*\left(\factor{A}{\ideal a}\right) = \Lambda(\bar u_1, \dots, \bar u_{k+l}),
  \]
  where the classes \( \bar u_i \) may be represented by elements \( u_i \in A \) such that
  \[
     u_i + u_i^* = \begin{cases}
                   \reduced\lambda_i\reduced\lambda^*_i & \text{for \( i=1,\dots, k \) } \\
                   \reduced\mu_i                        & \text{for \( i=k+1,\dots, k+l\) }
                   \end{cases}
  \]
\end{cor}
\begin{proof}
The ideal \(\ideal a\) may be generated by the reduced generators \(\rlambda,\rlambda_1^*\), \dots, \(\rlambda_k,\rlambda_k^*\), \(\rmu_{k+1}\), \dots, \(\rmu_{k+l}\) of \(R\). These generators form a regular sequence in \(R\). Since \(A\) is free of finite rank over \(R\), they also form a regular sequence in \(A\). So we may apply \autoref{cor:h-regular-sequence-multiplicative}.
\end{proof}

\vfill
\subsection*{Acknowledgements}
I am greatly indebted to Ian Grojnowski for initiating my search for a representation-theoretic description of the Witt rings of flag varieties, and for pointing me to the papers of Bousfield cited. In November 2010, he outlined a computational approach to me that differs from the one presented here; I have learned since that it lead him to identical results shortly afterwards. Nobuaki Yagita has been highly helpful in communicating the results of \cite{KO:ExceptionalFlags} and sending me an early version of \cite{Yagita:W-of-G}, and I am grateful to him for explaining some of his arguments. I also gratefully acknowledge support from the Max-Planck-Institute for Mathematics in Bonn. Substantial parts of this paper were worked out during the time I spent there in spring 2012. Finally, I thank Jeremiah Heller for first aid in times of confusion, Baptiste Calm{\`e}s for many interesting discussions and Jens Hornbostel for useful comments on an earlier version of this paper.

\end{document}